\documentclass[a4paper,11pt]{article}

\usepackage{longtable,geometry}
\usepackage{dsfont}
\usepackage[english]{babel}
\usepackage{amsmath}
\usepackage{amsfonts}
\usepackage{amssymb}
\usepackage{amsthm}
\usepackage{graphicx}
\usepackage{color}
\usepackage{hyperref}
\def\esp{\mathbb{E}}
\def\pr{\mathbb{P}}
\def\var{\mathrm{Var}}
\def\N{\mathbb{N}}
\def\F{\mathcal{F}_n}
\def\R{\mathbb{R}}

\def\x{\boldsymbol{x}}

\def\TV{\tiny \textsc{TV}}
\newtheorem{thm}{Theorem}[section]
\newtheorem{lem}{Lemma}[section]

\newtheorem{exe}{Example}[section]
\newtheorem{rem}{Remark}[section]

\newtheorem{prop}{Proposition}[section]
\newtheorem{defi}{Definition}[section]
\newtheorem*{assum}{Assumptions}
\newcommand{\dd}{\mathop{}\!\mathrm{d}}
\begin{document}
 
\title{Slow recurrent regimes for a class of one-dimensional stochastic growth models}
\author{Etienne Adam \thanks{CMAP, Ecole Polytechnique, CNRS, Universit\'e Paris-Saclay, route de Saclay, 91128 Palaiseau. France. Email : etienne.adam@polytechnique.edu}}

\maketitle
\begin{abstract}
 We classify the possible behaviors of a class of one-dimensional stochastic recurrent growth models. In our main result, we obtain nearly optimal bounds for the tail of hitting times of some compact sets. If the process is an aperiodic irreducible Markov chain, we determine whether it is null recurrent or positive recurrent and in the latter case, we obtain a subgeometric convergence of its transition kernel to its invariant measure. We apply our results in particular to state-dependent Galton-Watson processes and we give precise estimates of the tail of the extinction time.
\end{abstract}
\tableofcontents
\section{Introduction and main result}
\subsection{Introduction}
We consider a stochastic growth model $(X_n)_{n \in \N}$, taking values in $\mathcal{X}$, an unbounded subset of $\R_+$, and satisfying a stochastic difference equation of the form
\begin{equation}\label{stochasticdifference}
 X_{n+1}=X_n+g\left(X_n\right)+\xi_n,
\end{equation}
where $g$ is a given function and $(\xi_n)_{n\in\N}$ is a sequence of random variables such that almost surely, 
\[
\esp\left( \xi_n \big| \F\right)=0,
\]
\[
\esp\left(\xi_n^2 \big| \F \right)=\sigma^2\left(X_n\right)< \infty,
\]
for some positive function $\sigma^2(x)$. The filtration  $(\F)_{n \in \N}$ is such that $(X_n)_{n \in \N}$ is $\F$ measurable for all $n\in\N$.

Provided that the following limit exists
\[
      \theta=\underset{x \rightarrow \infty}{\lim} \frac{2xg\left(x\right)}{\sigma^2\left(x\right)},
\]
and belongs to $(-\infty,1)$, Kersting proved in \cite{kersting} that $\pr(\{X_n \underset{n \rightarrow \infty}{\longrightarrow}\infty\})=0$ and said that $\left(X_n\right)_{n\in \N}$ is recurrent, adopting the terminology from Markov chain theory, whereas if $\theta \in (1,\infty)$ then $\pr(\{X_n \underset{n \rightarrow \infty}{\longrightarrow}\infty\})>0$. A similar criterion for the multidimensional case was recently given in \cite{adam}.

The aim of this article is to determine how quickly the process $\left(X_n\right)_{n\in \N}$, started from $x>A$, goes into the interval $[0,A]$, where $A>0$ is arbitrary. If $\left(X_n\right)_{n\in \N}$ is an aperiodic irreducible Markov chain, we deduce therefrom a criterion of positive recurrence and how fast $\left(X_n\right)_{n\in \N}$ converges to its invariant measure. Moreover, if we have in mind population models,  where a natural assumption is the dichotomy property, \textit{i.e.},
\[
\pr\left(\left\{X_n \underset{n \rightarrow \infty}{\longrightarrow}\infty\right\}\right)+\pr\left(\left\{\exists n \text{ such that }X_n=0\right\}\right)=1, 
\]
we obtain precise estimates of the tail of the extinction time.

The first key ingredient of this article is to consider power functions as Lyapunov functions for growth models. Kersting \cite{kersting} proved recurrence and transience of growth models by using the logarithm as a Lyapunov function. However, we cannot get more information on the behaviour of $\left(X_n\right)_{n \in \N}$ with this function. Considering power functions yields an inequality of the form 
\[
 \esp\left(X_{n+1}^\alpha \big| \F\right) -X_n^\alpha \leq -C X_n^{\alpha-1}g(X_n)+b\mathds{1}_{\{X_n\leq A\}},
\]
for all $n \in \N$, where $\alpha \in (0,1)$, $A,C$ and $b$ some positive constants. From this equation, we deduce that
\begin{equation}\label{keyequation}
 \esp\left(f(Y_{n+1}) \big| \F\right) -f(Y_n) \leq -Cf'(Y_n)+b\mathds{1}_{\{Y_n\leq A\}},
\end{equation}
where $Y_n$ is a transform of $X_n$, $f$ is an increasing function, $A,C$ and $b$ some positive constants. Inequality \eqref{keyequation} enables us to give all possible behaviors of our class of recurrent growth models. In a series of papers \cite{aspandiiarov2, aspandiiarov3, aspandiiarov}, Aspandiiarov and al. proved upper and lower bounds for the tail of hitting-time into compact sets, for processes verifying some conditions, improving previous results of Lamperti \cite{lamperti2}. The second key ingredient, is to apply these results on a transform  $Y_n=G(X_n)$ of our process to get an upper bound of hitting-time into compact sets. If $(X_n)_{n\in\N}$ is an aperiodic irreducible Markov chain, we give a criterion for null recurrence or positive recurrence. Moreover, if $(X_n)_{n\in \N}$ is positive recurrent, we obtain from \cite{aspandiiarov3} in the countable state space, from \cite{douc} in a general state space, subgeometric rate of convergence to its invariant probability measure. Thus, we give a complete classification of behaviours of stochastic recurrent growth processes of the form \ref{stochasticdifference}. By applying our results, we deduce nearly optimal upper and lower bounds of the tail of the extinction time of state-dependent Galton-Watson processes that seem to have never been studied before, to the best of our knowledge. We also recover a weaker version of results of Zubkov \cite{zub} on the return time to zero of critical Galton-Watson process with immigration, but without using probability generating functions.

The article is organized as follows. In the next subsection, our main results Theorem \ref{bound} and Theorem \ref{markov} are stated. Then, in Section 2 we state and prove a series of lemmas needed for the proof of Theorems \ref{bound} and \ref{markov}. Section 3 is devoted to the proof of Theorem \ref{bound}. In Section 4, we consider that $(X_n)_{n\in \N}$ is an aperiodic irreducible Markov chain and we prove Theorem \ref{markov}. In section 5, we give various examples, in particular extinction time of state-dependent Galton-Watson processes. In the last section, we prove a key lemma for the lower bound of Theorem \ref{bound} and we recall some results from \cite{aspandiiarov3} that we use throughout this article.

\subsection{Main results}
We list the assumptions we need to formulate our main results.
\begin{assum}
\leavevmode\\
(A1) The function $g$ is positive, differentiable and $g\left(x\right)=o\left(x\right)$ when $x$ tends to infinity.\\
(A2) There exist $M>0$, $c_1>0$ and $\varepsilon>0$, such that for all $x>M$, for all $y>\left(1-\varepsilon\right)x$, 
\begin{equation}\label{xg}
xg\left(x\right)\leq c_1 yg\left(y\right).
\end{equation}
(A3)There exists $\delta>0$ such that for all $n \in \N$, $\esp\left(|\xi_n|^{4+\delta}\big| \F\right)\leq C\sigma^{4+\delta}\left(X_n\right)$.
\end{assum}
Let us comment on these assumptions.\\
Assumption (A1) precludes $X_n$ from growing geometrically, we focus on a kind of critical case where $X_n$ is perturbed by a drift $g\left(X_n\right)$. If the function $g$ is defined on a discrete subset of $\R_+$ then we consider a differentiable continuation of $g$. Assumption (A2) is rather technical, it encodes a non-decreasing property for the function $xg\left(x\right)$. We use it in Section 6 for the proof of Lemma \ref{verif}. If we consider $g\left(x\right)=x^\alpha$, then this simply means that $\alpha \in [-1,1)$. In \cite{kersting}, Kersting needs the existence of $2+\delta$-moments, to prove the recurrence of $X_n$. For technical reasons, detailed in Remark \ref{delta2} in Section \ref{section:lemma-proof}, we need the existence of $(4+\delta)$-moments to obtain the bounds of Theorem \ref{bound}.\\
Before stating the theorem, we introduce two transforms. Let 
\[
 G\left(x\right)=\int_1^x \frac{\dd y}{g\left(y\right)},
\]
and for $\alpha>0$, let
\[
 \ell_\alpha=\left(G^{-1}\left(x\right)\right)^\alpha.
\]
\begin{thm}\label{bound}
Besides (A1), (A2), (A3),  assume that there exist $\lambda>0$ and $\theta \in \left(0,1\right)$ such that 
\begin{equation}\label{g'g}
\underset{x \rightarrow \infty}{\lim}\frac{g'\left(x\right)x}{g\left(x\right)}=1-\lambda
\end{equation}
and
\begin{equation}\label{deftheta}
\underset{x \rightarrow \infty}{\lim} \frac{2xg\left(x\right)}{\sigma^2\left(x\right)}=\theta. 
\end{equation}
Then, there exists $A>0$ such that for all $x_0\in \mathcal{X} \cap (A, \infty)$, for all $\alpha$ and $\beta$ such that $0<\alpha<1-\theta<\beta$, there exist two constants $C_\alpha$ and $C_\beta$ such that for all $n \in \N$, 
\begin{equation}\label{thm}
 \frac{C_\beta}{\ell_\beta\left(n\right)}\leq \pr_{x_0}\left(\tau_A>n\right)\leq \frac{C_\alpha}{\ell_\alpha\left(n\right)},
\end{equation}
with $\tau_A=\inf\left\{n \in \N : X_n\leq A\right\}$.
\end{thm}
\begin{rem}
We prove the upper bound in \eqref{thm} by showing that $\esp_{x_0}(\ell_\alpha(\tau_A))<\infty$ for all $0<\alpha<1-\theta$ and $x_0 \in \mathcal{X}$. An easy consequence of this lower bound is that $\esp_{x_0}(\ell_\beta(\tau_A))=\infty$ for all $\beta>1-\theta$ and $x_0\in \mathcal{X} \cap (A, \infty)$. We cannot determine if $\esp_{x_0}(\ell_{1-\theta}(\tau_A))$ is finite or not.
\end{rem}
\begin{rem}
 In the proof of the theorem, we get explicit constants $C_\alpha$ and $C_\beta$ and in particular, the dependence of these constants on $x_0$.
\end{rem}
If $(X_n)_{n\in \N}$ is an aperiodic irreducible Markov chain, we determine when it is positive recurrent and the rate of convergence to the invariant probability measure. We denote by $P\left(.,.\right)$ the transition kernel of the Markov chain $\left(X_n\right)_{n\in \N}$. We deal with both countable state space and general state space. 
\begin{assum}
\leavevmode\\
(A4) $(X_n)_{n \in \N}$ is an aperiodic irreducible Markov chain taking values in a countable set $\mathcal{X}\subset \R_+$, such that for all $A>0$, $[0,A] \cap \mathcal{X}$ is finite.\\
(A4') $(X_n)_{n \in \N}$ is an aperiodic $\psi$-irreducible Markov chain taking values in a general state space $\mathcal{X}\subset \R_+$ and level sets $[0,A] \cap \mathcal{X}$ are petite sets for all $A>0$.
\end{assum}
We recall the definition of $\psi$-irreducibility (see \cite[p.84]{meyn}) :\\
We say that a Markov chain $\left(X_n\right)_{n\in \N}$ is $\psi$-irreducible if there exists a non trivial measure $\psi$ such that for all set $K\subset \mathcal{X}$,
\begin{equation}\label{irreductible}
 \psi\left(K\right)>0 \Rightarrow \pr_x\left(\exists n \text{ such that } X_n \in K\right)>0,
\end{equation}
and for all measures $\varphi$ satisfying \eqref{irreductible}, $\varphi$ is absolutely continuous with respect to $\psi$.
\begin{thm}\label{markov}
 Let us assume (A1), (A2), (A3) and (A4) or (A4') hold. Let $\lambda$ and $\theta$ be as in Theorem \ref{bound}.\\
Then $(X_n)_{n\in\N}$ is Harris-recurrent. Moreover \\
i) If $\lambda>1-\theta$, then $(X_n)_{n \in \N}$ is null recurrent.\\
ii) If $\lambda<1-\theta$, then $(X_n)_{n \in \N}$ is positive recurrent. Denote by $\pi$ its invariant probability measure. Then for all $\alpha \in (\lambda,1-\theta)$ and, if (A4) holds then for all probability measure $\nu$ on $\mathcal{X}$ such that 
\[
 \esp_\nu(\ell_\alpha'(\tau_A))<\infty,
\]
we have
\begin{equation}\label{markovcountable}
 \underset{n \rightarrow \infty}{\lim}\ell_\alpha'(n) \left\| \nu P^n - \pi \right\|_{\text{\TV}}=0,
\end{equation}
and if (A4') holds then for all $x\in \mathcal{X}$
\begin{equation}\label{markovgeneral}
\underset{n \rightarrow \infty}{\lim} \ell_\alpha'(n) \left\| P^n\left(x,.\right) - \pi(.) \right\|_{\text{\TV}}=0.
\end{equation}
\end{thm}

\begin{rem}
 The case $\lambda=1-\theta$ seems to be never treated, to the best of our knowledge.
\end{rem}
\begin{exe}
We consider a stochastic growth model defined by the stochastic difference equation \eqref{stochasticdifference} 
\[
 X_{n+1}=X_n+cX_n^\gamma+\xi_n
\]
with $\gamma \in (-1,1)$, $c>0$ and $\sigma^2(X_n)=\esp(\xi_n^2 \big| \F)=dX_n^{1+\gamma}$  with $d>0$. Then 
\begin{itemize}
 \item $\theta=\frac{2c}{d}$
\item $\lambda=1-\gamma$
\item $G(x) \propto x^{1-\gamma}$
\item $\ell_\alpha (x) \propto x^{\frac{\alpha}{1-\gamma}}$. 
\end{itemize}
By Theorem \ref{bound}, for all $\beta<1-\theta<\alpha$, there exists $A>0$ such that for all $x_0>A$, there exist $C_\beta>0$ and $C_\alpha>0$ such that 
\[
 \frac{C_\alpha}{n^{\frac{\alpha}{1-\gamma}}}\leq \pr_{x_0}\left(\tau_A>n\right)\leq \frac{C_\beta}{n^{\frac{\beta}{1-\gamma}}}.
\]
If $\gamma>\theta$ and $\left(X_n\right)$ is a Markov chain satisfying the assumptions of Theorem \ref{markov}, then $\left(X_n\right)$ is positive recurrent and for all $\alpha<1-\theta$, for all $x \in \mathcal{X} \subset \R_+$,
\[
\underset{n \rightarrow \infty}{\lim} n^{\frac{\alpha}{1-\gamma}-1} \| P^n(x,.) - \pi(.) \|_{\text{\TV}}=0, 
\]
where $\pi$ is the invariant probability measure of $(X_n)_{n \in \N}$.\\
If $c$ and $d$ are fixed, then by increasing $\gamma$, we make $(X_n)_{n\in \N}$ positive recurrent. Actually, the parameter $\gamma$ is related to both the drift $g(x)$ and the variance $\sigma^2(x)$, by increasing $\gamma$ we increase both of them but we can see that its effect on the variance is more important.
\end{exe}

\section{Preliminary results}

We state and prove here some important lemmas which will be useful for the proofs of theorems \ref{bound} and \ref{markov}. In the first lemma, we prove that $\left(X_{n\wedge \tau_A}^\alpha\right)_{n \in \N}$ is a supermartingale if $\alpha \in (0,1-\theta)$, and a submartingale if $\alpha \in (1-\theta,1)$.
\begin{lem}\label{martingale}
 Let us assume (A1) and (A3) and let $\lambda$ and $\theta$ be defined as in Theorem \ref{bound}.\\
i) If $\alpha \in \left(0,1-\theta\right)$, then there exist $A>0$, $C>0$ and $b>0$ such that for all $n \in \N$, 
\begin{equation}\label{lyapunov}
 \esp\left(X_{n+1}^\alpha \big| \F\right) \leq X_n^\alpha -Cg\left(X_n\right)X_n^{\alpha-1}+b \mathds{1}_{\{X_n \leq A\}} \text{ a.s.}
\end{equation}
ii) If $\alpha \in \left(1-\theta,1\right)$, then there exist $B>0$ and $b_1>0$ such that for all $n \in \N$
\begin{equation}\label{submartingale}
 \esp\left(X_{n+1}^\alpha \big| \F\right) \geq X_n^\alpha-b_1 \mathds{1}_{\{X_n \leq B\}} \text{ a.s.}
\end{equation}
\end{lem}
\begin{proof}
 For $D>0$ large enough, we have 
\begin{equation}\label{ineg}
 \left(1+u\right)^\alpha \leq 1+\alpha u + \frac{\alpha(\alpha-1)}{2}\,u^2+ D|u|^{3},
\end{equation}
for all $u \in (-1,+\infty)$. We obtain, for all $n \in \N$, if $X_n>0$,
\begin{align*}
 \esp\left(X_{n+1}^\alpha \big|\F\right) & \leq \esp\left(X_n^\alpha\left(1+\frac{g(X_n)+\xi_n}{X_n}\right)^\alpha \Big| \F\right)\\
& \leq \esp\left(X_n^\alpha \left(1+\alpha  \left(\frac{g(X_n)+\xi_n}{X_n}\right) + \frac{\alpha(\alpha-1)}{2}\left(\frac{g(X_n)+\xi_n}{X_n}\right)^2\right) \Big|\F \right)\\
& \qquad +\esp\left(X_n^\alpha \left(D\left|\frac{g(X_n)+\xi_n}{X_n}\right|^{3}\right) \Big|\F\right)\\
& \leq X_n^\alpha+\alpha \left(g\left(X_n\right)X_n^{\alpha-1}-\frac{1-\alpha}{2}\sigma^2\left(X_n\right)X_n^{\alpha-2}\right) + R_n,
\end{align*}
with
\[
 R_n=\frac{\alpha(\alpha-1)}{2}\,g(X_n)^2X_n^{\alpha-2}+D\esp\left(|g(X_n)+\xi_n|^{3} X_n^{\alpha-3} \big|\F\right). 
\]
By H\"older's inequality and (A3), 
\[
 R_n \leq  \frac{\alpha(\alpha-1)}{2}\,g(X_n)^2X_n^{\alpha-2}+D'|g(X_n)|^3X_n^{\alpha-3}+D''\sigma^3(X_n)X_n^{\alpha-3}.
\]
But $\sigma^3(x)x^{\alpha-3}=o\left(g(x)x^{\alpha-1}\right)$ when $x$ tends to infinity, then there exist $C, B, b>0$ such that
\[
 \esp\left(X_{n+1}^\alpha \big|\F\right) \leq X_n^\alpha -Cg\left(X_n\right)X_n^{\alpha-1}+b \mathds{1}_{\{X_n \leq B\}}.
\]
Since there exists a positive constant $D$ such that 
\[
\left(1+u\right)^\alpha \geq 1+\alpha u + \frac{\alpha(\alpha-1)}{2}u^2- D|u|^{2+\delta},
\]
for all $u \in (-1,+\infty)$, the proof of \eqref{submartingale} is similar.
\end{proof}

The two first statements of the next lemma, on the top of the previous one, give us a better understanding of the criterion of Theorem \ref{markov}, \textit{i.e.}, the comparison between $\lambda$ and $1-\theta$. Some points of this lemma are stated and proved in \cite{kerstinggamma}.

\begin{lem}\label{lemmapre}
\leavevmode\\
Let us assume (A1) and let $\lambda$ be defined as in \eqref{g'g}.
\begin{enumerate} 
\item \label{concave} For $\alpha \in \left(0, \lambda\right)$, the function $ \ell_\alpha$ is ultimately concave.
\item \label{convex} For $\alpha \in \left(\lambda, +\infty\right)$, the function $\ell_\alpha$ is ultimately convex.
\item \label{Gg/xitem} We have
 \begin{equation}\label{Gg/x}
\underset{x \rightarrow \infty}{\lim}\, \frac{x}{G\left(x\right)g\left(x\right)}=\lambda.
 \end{equation}
\item \label{g0item} For all $\mu<\lambda$,
\begin{equation}\label{gO}
 g(x)=\mathcal{O}(x^{1-\mu}),
\end{equation}
when $x$ tends to infinity.
\item \label{rlxitem} Let $\alpha >0$, for all $r \in \R_+^*$, there exists a positive constant $A_r$ such that for all $x \in \R_+$,
\begin{equation}\label{rlx}
 A_r\ell_\alpha(x) \geq \ell_\alpha(rx).
\end{equation} 
\end{enumerate}
\end{lem}

\begin{proof}
We first prove statements \ref{concave} and \ref{convex}. We recall that a function $f$ is ultimately concave or ultimately convex if there exists $x_0>0$ such that the restriction of $f$ to $[x_0,\infty)$ is concave or convex respectively.
The second derivative of $\ell_\alpha$ is 
\[
 \alpha g(G^{-1}(x))\,G^{-1}(x)^{\alpha-2}\left(g'(G^{-1}(x))\,G^{-1}(x)+(\alpha-1)\,g(G^{-1}(x))\right).
\]
By the substitution $u=G^{-1}(x)$, we obtain that for $x$ large enough, $\ell_\alpha''(x)<0$ if $\alpha \in (0, \lambda)$ and $\ell_\alpha''(x)>0$ if $\alpha \in (\lambda, +\infty)$.\\
We defer to \cite{kerstinggamma} for the proof of the third statement.\\
Let us now prove the fourth one. Since 
\[
 \underset{x \rightarrow \infty}{\lim}\frac{g'\left(x\right)x}{g\left(x\right)}=1-\lambda,
\]
we have that for all $\varepsilon>0$ there exists a constant $M$ such that for all $x>M$,
\[
 \frac{g'(x)}{g(x)}\leq \frac{1-\lambda+\varepsilon}{x}.
\]
By integrating this inequality between $M$ and $v>M$, we obtain that 
\begin{equation}\label{gvgM}
 g(v)\leq \left(\frac{v}{M}\right)^{1-\lambda+\varepsilon}g(M),
\end{equation}
thus $g(x)=\mathcal{O}(x^{1-\mu})$ for all $\mu<\lambda$.\\
Finally, we prove the laste statement. If $r \leq 1$, then, since $\ell_\alpha$ is an increasing function, we take $A_r=1$.\\
We now assume that $r>1$. Let $A_r>1$ to be fixed later on. By \eqref{gvgM}, we have for large $x$
\begin{equation}\label{gx}
 g(A_rx)\leq A_r^{1-\lambda/2}g(x).
\end{equation}
By \eqref{Gg/x}, we know that for $x$ large enough
\[
 \lambda/2 \leq \frac{x}{G(x)g(x)} \leq 2\lambda.
\]
Applying this inequality twice, for $x$ and $A_rx$, and by \eqref{gx}, we obtain
\begin{align*}
 2\lambda G(A_rx) & \geq \frac{A_rx}{g(A_rx)} \geq \frac{x A_r^{\lambda/2}}{g(x)}\\
& \geq G(x) A_r^{\lambda/2} \lambda/2.
\end{align*}
We set $A_r=(4r)^{2/\lambda}$ and then, for all $x$ sufficiently large, we get
\[
 rG(x)\leq G(A_rx).
\]
Let $y=G(x)$ and let compose the last inequality by $\ell_1$, which is an increasing function, 
\[
 \ell_1(ry) \leq A_r \ell_1(y).
\]
Rising this inequality to the power $\alpha$ yields \eqref{rlx}.
\end{proof}
\section{Polynomial asymptotics of the tail of hitting times}
The aim of this section is to prove Theorem \ref{bound}.\\
We first prove the upper bound of the inequality \eqref{thm} by using Theorem 2 and Theorem 3 in \cite{aspandiiarov3}, we recall them in the last section.
Let $\mathcal{A}$ be the set of positive function $f$ such that there exists a positive constant $A_f$ such that 
\[
 \underset{x \rightarrow \infty}{\limsup} \frac{f(2x)}{f(x)}\leq A_f.
\]
For all real valued functions $h$, let $\mathcal{B}_h$ be the set of  positive functions $f \in \mathcal{C}^2(0, \infty)$ ultimately concave, such that $\lim_{x \rightarrow \infty}f(x)=\infty$, $\lim_{x \rightarrow \infty}f'(x)=0$, and such that the integral
\begin{equation}\label{integral}
\int_1^\infty \frac{f'(x)\dd x}{h\circ r(x)} \quad \text{ converges,}
\end{equation}
with $r(x)=\sup\{y \geq A, f'(x)=h'(y)\}$.

\begin{prop}\label{upperboundconcave}
 We assume (A1), (A3) and that $\lambda$ and $\theta$ are defined as in Theorem \ref{bound}. There exists $A>0$ such that for all $x_0>A$, $\gamma$ and $\eta$ such that $\gamma<\eta<1-\theta$, then there exists a constant $K(\gamma,\eta)$ such that for all $n \in \N$,
\[
 \pr_{x_0}(\tau_A>n)\leq \frac{K(\gamma,\eta)x_0^\eta}{\ell_\gamma(n)}.
\]
\end{prop}
\begin{proof}
If $\gamma>\lambda$, then $\ell_\gamma$ is ultimately convex. We know by \eqref{rlx} that $\ell_\gamma \in \mathcal{A}$ and then we apply Theorem \ref{aspan2} and get the upper bound by Chebyshev's inequality.\\
If $\eta<\lambda$, then $\ell_\eta$ is ultimately concave. To apply Theorem \ref{aspan3} with $f=\ell_\gamma$ and $h=\ell_\eta$, we need also to check that the integral \eqref{integral} converges. Let $r(x)=\sup\{y \geq A, \ell_\gamma'(x)=\ell_\eta'(y)\}$. We first prove that for $x$ large enough, we have $x \leq r(x)$.\\
We recall that $\ell_\gamma'(x)=\gamma g\left(G^{-1}(x)\right) \left(G^{-1}(x)\right)^{\gamma-1}$. Thus, 
\[
\frac{\ell_\gamma'(x)}{\ell_\eta'(x)}=\frac{\gamma}{\eta}\left(G^{-1}(x)\right)^{\gamma-\eta} \underset{x \rightarrow \infty}{\longrightarrow} 0.
\]
Since $G^{-1}(x)$ increases to infinity, there exists $A_1>0$ such that for all $x>A_1$, $\ell_\gamma'(x)\leq \ell_\eta'(x)$ and then, for all $x>A_1$, $r(x)\geq x$.\\
Since $\ell_\eta$ is an increasing function, we obtain by substitution
\begin{align*}
\int^\infty \frac{\ell_\gamma'(x)\dd x}{\ell_\eta \circ r(x)} & \leq \int^\infty \frac{\ell_\gamma'(x)\dd x}{\ell_\eta(x)}\\
& \leq \int^\infty \frac{\ell_\gamma'(x)\dd x}{(\ell_\gamma (x))^{\eta/\gamma}} \\
& \leq C\int^\infty \frac{\dd u}{u^{\eta/\gamma}}<\infty.
\end{align*}
Finally, we obtain the upper bound by Chebyshev's inequality.
\end{proof}

Before proving the lower bound of Theorem \ref{bound}, we recall an important lemma from \cite{aspandiiarov}:
\begin{lem}[\cite{aspandiiarov}, Lemma 2]\label{lemfonda}
 Let $Y_n$ be a $\F$-adapted stochastic process taking values in an unbounded subset of $\R_+$. Suppose there exist positive constants $A$, $C$ and $D$ such that for all $n \in \N$, on $\{\tau_A>n\}$,
\[
 \esp(Y_{n+1}-Y_n \big|\F) \geq -C 
\]
and, for some $r>1$,
\[
 \esp\left(Y_{n+1}^r-Y_n^r \big| \F \right)\leq DY_n^{r-1}. 
\]
Then, for any $\nu \in (0,1)$, there exist positive $\varepsilon$ and $d$ that do not depend on $A$ such that for any $n\in \N$, on $\{Y_{n\wedge \tau_A}>A(1+d)\}$,
\[
 \pr\left(\tau_A>n+\varepsilon Y_{n \wedge \tau_A} \right)\geq 1-\nu.
\]
\end{lem}
The next lemma is crucial. We defer its proof, which is rather technical, to Section \ref{section:lemma-proof}:
\begin{lem}\label{verif}
 For all $n \in \N$, let $Y_n=G(X_n)$. We assume (A1), (A2), (A3) and that $\lambda$ and $\theta$ are defined as in Theorem \ref{bound}. Then $(Y_n)_{n\in \N}$ satisfies Lemma \ref{lemfonda}.
\end{lem}
\begin{prop}\label{lowerbound}
 We assume (A1), (A2) and (A3). Let $\beta>1-\theta$. There exists $A>0$ such that for all $x_0>A$, there exists $\varepsilon_0>0$ and $C>0$ such that for all $n \in \N$,
\[
 \pr_{x_0}\left(\tau_A>n\right)\geq C\frac{x_0^\beta-A^\beta}{\ell_\beta(n/\varepsilon_0)}.
\]
\end{prop}
\begin{proof}
The proof of the lower bound is as follows: we know by Lemma \ref{verif} that $Y_n$ verifies Lemma \ref{lemfonda} and then we follow the proof of Theorem 1 in \cite{aspandiiarov2}. We relax the assumption of bounded jumps of this theorem by using H\"older's inequality.\\
Let $\beta>1-\theta$.\\
By Lemma \ref{verif}, we know that Lemma \ref{lemfonda} applies to $Y_n=G(X_n)$.
By Lemmas \ref{lemfonda} and \ref{verif}, there exist $\varepsilon_0>0$ and $d>0$ such that for any $n$ : 
\[
 \pr\left(\tau_A>n+\varepsilon_0 Y_{n\wedge \tau_A}\big|\F\right)\geq 1-v \qquad \text{on } \left\{Y_{n\wedge \tau_A}>G(A)(1+d)\right\}.
\]
This implies that for any stopping time $\mu$ we have 
\[
 \pr\left(\tau_A>\mu+\varepsilon_0 Y_{\mu \wedge \tau_A}\big|\mathcal{F}_\mu\right)\geq 1-v \qquad \text{on } \left\{Y_{\mu \wedge \tau_A}>G(A)(1+d)\right\}\cap \left\{\mu<\infty\right\}.
\]
For each $S>0$, let 
\[
 \tilde{\tau}_S=\inf\left\{n\geq 0, Y_n\geq S\right\}.
\]
Let us fix $B$ such that $B>G(A)(1+d)$.\\
Then, 
\begin{align}
 \pr\left(\tau_A\geq \varepsilon_0 B\right) & \geq \pr\left(\tau_A>\tilde{\tau}_B+\varepsilon_0 Y_{\tilde{\tau}_B \wedge \tau_A}, \tilde{\tau}_B<\tau_A\right)\nonumber\\
& = \esp\left(\mathds{1}_{\{\tilde{\tau}_B<\tau_A\}}\pr\left(\tau_A>\tilde{\tau}_B+\varepsilon_0 Y_{\tilde{\tau}_B\wedge \tau_A} \big| \mathcal{F}_{\tilde{\tau}_B}\right)\right) \nonumber\\
& \geq (1-v)\pr\left(\tilde{\tau}_B<\tau_A\right). \label{tauAtauB}
\end{align}
Since $\left(\tau_A \wedge \tilde{\tau}_B\right)<\infty$ and $\ell_\beta\left(Y_{n\wedge \tau_A \wedge \tilde{\tau}_B}\right)$ is a submartingale by Lemma \ref{martingale}, we have 
\[
 x_0^\beta=\ell_\beta(Y_0)\leq \esp\left(\ell_\beta(Y_{\tau_A \wedge \tilde{\tau}_B})\right).
\]
Since $g(x)=o(x)$, there exists $K>0$ such that $\esp\left(\ell_1\left(Y_{\tilde{\tau}_B-1}\right)+g\left(X_{\tilde{\tau}_B-1}\right)\right)\leq K^{1/\beta}\esp\left(X_{\tilde{\tau}_B-1}\right)$ and then 
\begin{align}
\esp\left(\ell_\beta\left(Y_{\tilde{\tau}_B}\right)\mathds{1}_{\{\tilde{\tau}_B<\tau_A\}}\right)& \leq \esp\left(\ell_{1}\left(Y_{\tilde{\tau}_B}\right)\right)^{\beta}\pr\left(\tilde{\tau}_B<\tau_A\right) \nonumber\\
& \leq \esp \left( \esp \left( \ell_{1} \left( Y_{\tilde{\tau}_B }\right)\big|\mathcal{F}_{\tilde{\tau}_B-1}\right)\right)^{\beta}\pr\left(\tilde{\tau}_B<\tau_A\right)\nonumber\\
& \leq \esp\left(\ell_1\left(Y_{\tilde{\tau}_B-1}\right)+g\left(X_{\tilde{\tau}_B-1}\right)\right)^\beta \pr\left(\tilde{\tau}_B<\tau_A\right) \nonumber\\
& \leq K\ell_\beta\left(B\right)\pr\left(\tilde{\tau}_B<\tau_A\right) \label{ineq}.
\end{align}
Hence,
\begin{align*}
  x_0^\beta & \leq \esp\left(\ell_\beta\left(Y_{\tau_A}\right)\mathds{1}_{\{\tilde{\tau}_B>\tau_A\}}\right)+\esp\left(\ell_\beta\left(Y_{\tilde{\tau}_B}\right)\mathds{1}_{\{\tilde{\tau}_B<\tau_A\}}\right)\\
& \leq \ell_\beta\left(G(A)\right)+K\ell_\beta\left(B\right) \pr\left(\tilde{\tau}_B<\tau_A\right),
\end{align*}
by \eqref{ineq} and
\[
 \pr\left(\tilde{\tau}_B<\tau_A\right)\geq \frac{ x_0^\beta-\ell_\beta\left(G(A)\right)}{Kl_\beta\left(B\right)}.
\]
Then, by \eqref{tauAtauB}, for $n>\varepsilon_0 G(A)\left(1+\delta\right)$, 
\begin{equation}\label{ineqepsilon}
 \pr\left(\tau_A>n\right)\geq \left(1-v\right) \frac{ x_0^\beta-A^\beta}{K\ell_\beta\left(n/\varepsilon_0\right)}.
\end{equation}
\end{proof}

\begin{proof}[Proof of Theorem \ref{bound}]
The upper bound is a direct consequence of Proposition \ref{upperboundconcave}. The lower bound comes from Proposition \ref{lowerbound} and \eqref{rlx}.
\end{proof}
\section{The Markov case : subgeometric rate of convergence}
In this section, we prove Theorem \ref{markov}, firstly the countable state space case and secondly the general state space case. We apply some results from \cite{aspandiiarov3} that we recall in the last section.\\
Let $\mathcal{G}$ be the set of positive functions $f$ such that there exist a positive function $h$ such that $h(x)\rightarrow 0$ as $x \rightarrow \infty$ and a positive constant $c$ such that for any positive $m\geq 1$, $x_1\geq 1, \ldots, x_m \geq 1$, 
\[
\displaystyle{f\left(\sum_{k=1}^mx_k\right)\leq c e^{mh(m)}\sum_{k=1}^mf(x_k)}.
\]
Let $\mathcal{G}'$ be the set of non decreasing in a neighborhood of infinity functions $f$ such that $\ln(f(x))/x$ is non increasing in a neighborhood of infinity and tends to zero when $x$ tends to infinity.

\begin{proof}[Proof of Theorem \ref{markov} for a countable state space]
Let $A$ be defined as in Theorem \ref{bound}. We know by (A4) that $F=[0,A] \cap \mathcal{X}$ is finite. First note that for all $z \in F$, by Markov property we have
\begin{equation}\label{hitting}
 \esp_z(\tau_F)=\pr_z(X_1 \in F)+\sum_{s \in \mathcal{X}\setminus F} \pr_z(X_1=s)\esp_s(\tau_F).
\end{equation}
i) Let us assume that $\lambda>1-\theta$. We prove that for all $s \in \mathcal{X} \setminus F$, $\esp_s(\tau_F)=\infty$. Let $\beta \in (1-\theta,\lambda)$. By Theorem \ref{bound} we know that if $\sum 1/\ell_\beta(n)$ diverges, then $\esp_s(\tau_F)=\infty$, for all $s \in \mathcal{X} \setminus F$. The sum $\sum 1/\ell_\beta(n)$, is of the same nature that the integral $\int \dd x/\ell_\beta(x)$. By the substitution $u=G^{-1}(x)$, we obtain
\[
\displaystyle{\int_.^\infty \frac{\dd x}{\ell_\beta(x)}=\int_.^\infty \frac{\dd u}{u^\beta g(u)}= \infty},
\]
since $g(u) \leq Ku^{1-\lambda+(\lambda-\beta)/2}$.\\
Since $(X_n)_{n\in \N}$ is irreducible, there exists $(z_0,s_o) \in F\times \mathcal{X} \setminus F$ such that $\pr_{z_0}(X_1=s_0)>0$. Thus $\esp_{z_0}(\tau_F)=\infty$ and by Proposition \ref{hittingreturn}, $\esp_{z_0}(\tau)=\infty$, then $(X_n)_{n\in \N}$ is null recurrent.\\
ii) Let us assume that $\lambda<1-\theta$. Let $\eta \in (\lambda, 1-\theta)$. We first prove that there exists a positive constant $K$ such that for all $s \in \mathcal{X} \setminus F$, $\esp_s(\tau_F) \leq Ks^\eta$. Let $\gamma \in (\lambda,\eta)$. By Proposition \ref{upperboundconcave}, we know that 
\[
 \pr_s(\tau_F>n)\leq \frac{K(\gamma,\eta)s^\eta}{\ell_\gamma(n)}.
\]
We check that $\sum_{n=1}^{\infty} 1/\ell_\gamma(n)<\infty$. Since $\ell_\eta$ is convex, there exists a constant $C$ such that 
\[
 \displaystyle{\sum_{n=1}^{\infty} \frac{1}{\ell_\gamma(n)}\leq C\sum_{n=1}^{\infty} \frac{\ell_\eta'(n)}{\ell_\gamma(n)}}.
\]
This series is of the same nature that the integral
\[
 \displaystyle{\int_.^{\infty} \frac{\ell_\eta'(x)\dd x}{\ell_\gamma(x)}=\int_.^{\infty} \frac{\ell_\eta'(x)\dd x}{\ell_\eta(x)^{\gamma/\eta}}=K\int_.^{\infty}\frac{\dd u}{u^{\gamma/\eta}}<\infty}.
\]
Thus 
\begin{equation}\label{finiteesperance}
 \esp_s(\tau_F) \leq Ks^\eta.
\end{equation}
By \eqref{finiteesperance} and \eqref{hitting}, we obtain
\[
 \esp_z(\tau_F)\leq 1+K \esp_z(X_1^\eta)<\infty,
\]
thus by Proposition \ref{hittingreturn}, for any $z \in F$, $\esp_z(\tau)<\infty$ so $(X_n)_{n \in \N}$ is positive recurrent.\\
Let $\alpha \in (\lambda,1-\theta)$ and $\beta \in (\alpha,1-\theta)$. To apply Theorem \ref{fphi} with $f=\ell_\alpha$ and $\phi=\ell_\beta$, we need to check that $\ell_\alpha \in \mathcal{G}$ and $\ell_\alpha' \in \mathcal{G}'$. Since $\ell_\alpha$ is convex, we have for all $m \geq 1$, $\x_1\geq 1, \ldots, x_m\geq1$
\[
 \displaystyle{\ell_\alpha\left(\sum_{k=1}^m x_k\right)\leq \frac{1}{m} \sum_{k=1}^m \ell_\alpha(mx_k)},
\]
and by \ref{rlxitem} of Lemma \ref{lemmapre},
\[
 \displaystyle{\ell_\alpha\left(\sum_{k=1}^m x_k\right)\leq \frac{(4m)^{2\alpha/\lambda}}{m} \sum_{k=1}^m \ell_\alpha(x_k)\leq 4^{2\alpha/\lambda}e^{(2\alpha/\lambda-1)\ln(m)} \sum_{k=1}^m \ell_\alpha(x_k)},
\]
thus $\ell_\alpha \in \mathcal{G}$.\\
We recall that $\ell_\alpha'(x)=\alpha g(G^{-1}(x))(G^{-1}(x))^{\alpha-1}$. Since $G^{-1}(x) \rightarrow \infty$ as $x \rightarrow \infty$ and $\alpha-1<0$, we only need to prove that 
\[
 \frac{\ln g(G^{-1}(x))}{x}\rightarrow 0 \text{ as } x \rightarrow \infty.
\]
By the substitution $u=G^{-1}(x)$ and since $g(x)=\mathcal{O}(x^{1-\mu})$ for all $\mu<\lambda$ by \ref{g0item} of Lemma \ref{lemmapre}, we obtain that $x^\mu=\mathcal{O}(G(x))$ and then $\ell_\alpha'(x)/\ln(x) \rightarrow 0$ as $x \rightarrow \infty$, so $\ell_\alpha' \in \mathcal{G}'$.
\end{proof}

In the general state space case, we use a drift condition which comes from \cite{douc} :
\begin{defi}
We say that the condition $\boldsymbol{D}\left(\phi,V,\Gamma\right)$ is verified if there exist a function $V$, a concave monotone non-decreasing differentiable function $\phi \,: [1,\infty] \mapsto (0,\infty]$, a measurable set $\Gamma$ and a finite constant $b$ such that for all $x \in \R_+$
\[
 \esp_x\left(V(X_{1}) \right)+ \phi \circ V(x) \leq V(x)+b \mathds{1}_{\{x \in \Gamma\}}.
\]
\end{defi}
\begin{prop}[\cite{douc}, Proposition 2.5]\label{prop}
Let $P$ be a $\psi$-irreducible and aperiodic kernel. Assume that $\boldsymbol{D}(\phi,V,\Gamma)$ holds for a function $\phi$ such that $\underset{t \rightarrow \infty}{\lim} \phi'(t)=0$, a petite set $\Gamma$ and a function $V$ such that $\{V<\infty\}\neq \emptyset$. Then, there exists an invariant probability measure $\pi$, and for all $x$ in the full and absorbing set $\{V<\infty\}$, \textit{i.e.} $\pi(\{V<\infty\})=1$, 
\[
 \underset{n \rightarrow \infty}{\lim}r_\phi(n)\left\| P^n(x,.) - \pi(.) \right\|_{\text{\TV}}=0,
\]
with $r_\phi\left(x\right)=\phi \circ \Phi^{-1}\left(x\right)$ and  $\displaystyle{\Phi\left(x\right)=\int_1^{x}\frac{\dd u}{\phi(u)}}$.
\end{prop}
The proof of Theorem \ref{markov} in the general state space case consists essentially in checking that the condition $\boldsymbol{D}\left(\phi,V,\Gamma\right)$ holds.\\
We also recall that a set $C$ is regular if for all set $B$ such that $\psi(B)>0$, 
\[
\underset{x \in C}{\sup} \,\esp_x(\tau_B)<\infty,
\]
where $\tau_B$ is the first hitting-time of the set $B$. A Markov chain is called regular if there exists a countable cover of $\mathcal{X}$ by regular sets.

\begin{proof}[Proof of Theorem \ref{markov} for a general state space]
Since $[0,A]\cap \mathcal{X}$ is petite and since for all $x \in \mathcal{X}$, $\pr_x(\tau_A<\infty)=1$, we know from \cite[Proposition 9.1.7 p.205]{meyn} that $(X_n)_{n\in \N}$ is Harris-recurrent.\\
i) We assume that $\lambda>1-\theta$. By Theorem \ref{bound}, $\forall x \in (A,\infty)\cap \mathcal{X}$, $\esp_x(\tau_A)=\infty$. We assume that $(X_n)_{n\in\N}$ is positive recurrent to get a contradiction. By \cite[Theorem 11.1.4 p.260]{meyn}, we know that there exists a decomposition $\mathcal{X}=\mathcal{S}\cup \mathcal{N}$ with $\mathcal{S}$ full and absorbing and $(X_n)_{n\in\N}$ restricted to $\mathcal{S}$ is regular. Since $\mathcal{S}$ is absorbing, we know that $[0,A] \cap \mathcal{S} \neq \emptyset$ and $(A,\infty) \cap \mathcal{S} \neq \emptyset$. Let $C \subset \mathcal{S}$ be a regular set  of the countable cover of $\mathcal{S}$ such that $C \cap (A,\infty)\neq \emptyset$. Then there exists $x \in C \cap (A,\infty)$, and we know that $\esp_x(\tau_{A})=\infty$ which is a contradiction with $C$ is regular. Then $(X_n)_{n\in \N}$ is not positive recurrent but null recurrent.\\
ii) We assume that $\lambda<1-\theta$. Let $\alpha\in (\lambda, 1-\theta)$ and $\phi(x)=g(x^{\frac{1}{\alpha}})\,x^{\frac{\alpha-1}{\alpha}}$. Using Lemma \ref{martingale},
\[
\esp_x\left(X_{1}^\alpha \right)\leq x^\alpha-C\phi \left( x^\alpha \right)+b \mathds{1}_{\{x\leq A\}}.
\]
We now prove that $\phi$ is a concave non-decreasing function.\\
We first calculate the derivative of function $\phi$ : 
\[
\phi'(x)=\frac{g'(x^{1/\alpha})}{\alpha}-\frac{1-\alpha}{\alpha x^{1/\alpha}}. 
\]
For large $x$, using \eqref{g'g}, we obtain 
\[
\phi'(x) = \left(\frac{1-\lambda}{\alpha} g(x^{1/\alpha})-\frac{1-\alpha}{\alpha}\right)x^{-1/\alpha}+o\left(\frac{g(x^{1/\alpha})}{x^{1/\alpha}}\right).
\]
From \eqref{g'g} and $1-\lambda>0$, we know that $g'$ is ultimately positive and that $g$ tends to infinity. Thus, $\phi'$ is ultimately positive, non-increasing and tends to zero when $x$ tends to infinity (because $g(x)=o(x)$). Thus, the condition $\boldsymbol{D}\left(\phi,V,\Gamma\right)$ holds.
By a short computation, we see that $r_\phi(x)=\ell_\alpha'(x)$. Since $[0,A]\cap \mathcal{X}$ is a petite set by assumption, we apply Proposition \ref{prop} and there exists an invariant probability measure $\pi$ such that for all $x$
\[
\underset{n \rightarrow \infty}{\lim} \ell_\alpha'(n) \left\| P^n(x,.) - \pi(.) \right\|_{\text{\TV}}=0.
\]
\end{proof}

\section{Examples and applications}
We now illustrate our results by applying Theorem \ref{bound} and Theorem \ref{markov} to several models.
\subsection{Bessel-like walks}
A Bessel-like walk is a random walk on $\N$, reflecting at $0$, with steps $\pm 1$ and transition probabilites of the form 
\[
 \pr(X_{n+1}=x+1 \big| X_n=x)=p_x=\frac{1}{2}\left(1-\frac{\delta}{2x}+o\left(\frac{1}{x}\right)\right)
\]
and 
\[
 \pr(X_{n+1}=x-1 \big| X_n=x)=1-p_x
\]
where $x\geq 1$, $\delta \in \R$ and the $o(1/x)$ holds for $x$ tending to infinity. A Bessel-like walk is recurrent if $\delta>-1$, positive recurrent if $\delta>1$ and transient if $\delta<-1$.\\
We assume here that $\delta \in (-1,0)$. There exists $A>0$ such that we obtain an estimation of the tail of the hitting-time of the compact set $[0,A]$. 
\begin{prop}
 For all $\alpha$,$\beta$ such that $\alpha<1+\delta<\beta$, there exists $A>0$ such that for all $x_0>A$, there exist two positive constants $C_\alpha$ and $C_\beta$ such that
\[
 \frac{C_\beta}{n^{\beta/2}}\leq \pr_{x_0}(\tau_A>n)\leq \frac{C_\alpha}{n^{\alpha/2}}.
\]
\end{prop}
For more precise results on Bessel-like walks and in particular asymptotic behaviours of $\pr_x(\tau_0>n)$ and $\pr_x(\tau_0=n)$, we defer to \cite{alexander}.
\subsection{Critical Galton-Watson process with immigration}
We consider a critical Galton-Watson process with immigration $(X_n)_{n \in \N}$ defined by
\[
X_{n+1}=\sum_{k=1}^{X_n} \xi_{k,n}+I_n,
\]
where $\left(\xi_{k,n}\right)_{k,n \in \N}$ are i.i.d. integer-valued random variables such that $\esp(\xi_{1,1})=1$, $\var(\xi_{1,1})=d>0$ and $\esp(\xi_{1,1}^{2+\delta})<\infty$ for some $\delta>0$ and i.i.d. integer-valued random variables $(I_n)_{n\in \N}$ such that $\esp(I_1)=c>0$, $\esp(I_{1}^{2+\delta})<\infty$ and the variables $(\xi_{k,n})_{k,n \in \N}$ and $(I_n)_{n \in \N}$ are independent.\\
Zubkov proved in \cite{zub} that the Markov chain $(X_n)_{n\in \N}$ is recurrent if $\theta=\frac{2c}{d}<1$ and gave the asymptotic behaviour of the tail of the return-time to zero $T_0=\inf\{n\geq 1 \text{ such that } X_n=0\}$: 
\[
\pr_0(T_0>n)\sim L(n)n^{\theta-1},
\]
with $L$ a slowly varying function. He also needed weaker moments assumptions.\\
We get here a weaker version of his result but without using neither the branching property nor probability generating functions.
\begin{prop}
There exists $A>0$, such that for all $x_0>A$, $\alpha,\beta$ such that $\alpha<1-\theta<\beta$, there exist some positive constants $C_\alpha$ and $C_\beta$ such that for all $n\in \N$ 
\[
\frac{C_\beta}{n^{\beta}} \leq \pr_{x_0}(\tau_A>n) \leq \frac{C_\alpha}{n^{\alpha}}.
\]
\end{prop}
\subsection{Extinction time of state-dependent Galton-Watson process}
State-dependent Galton-Watson processes were introduced by Klebaner in \cite{klebaner84} and H\"opfner in \cite{hopf}. They both gave condition for extinction and gamma-type limiting distribution for the process. However, to the best of our knowledge, extinction times of state-dependent Galton-Watson processes were never investigated.\\
Let $(X_n)_{n\in\N}$ be state-dependent Galton-Watson process defined as follows : 
\[
 X_{n+1}=\underset{k=1}{\overset{X_n}{\sum}}A_{k,n}(X_n),
\]
where $\esp(A_{k,n}(X_n) \big| X_n=x)=1+\frac{c}{x}$ and $\mathbb{V}ar(A_{k,n}(X_n) \big| X_n=x)=\sigma^2+o(1)$ with $c>0$ and $\sigma^2>0$. We assume that $0$ is an absorbing state and that for all $A>0$ and all $n \in \N$, there exist $\varepsilon>0$ and $k_A\in \N^*$, 
\begin{equation}\label{eteint}
 \pr(X_{n+k_A}=0 \big| X_n\leq A)\geq \varepsilon.
\end{equation}
This assumption implies the dichotomy property (see Theorem 3.1 in \cite{jagers}), that is to say,
\[
\pr\left(\left\{ \exists n \text{ such that }X_n=0\right\}\right)+\pr\left(\left\{X_n \underset{n \rightarrow \infty}{\longrightarrow}\infty\right\}\right)=1.
\]
We denote the extinction time  by $\tau_0= \inf\{ n\in \N \text{ such that } X_n=0\}$.
\begin{thm}
Let $\theta=\frac{2c}{\sigma^2}$ and assume that $\theta\in (0,1)$. Then, for all $\alpha<1-\theta<\beta$, for all $x \in \N^*$, there exist two constants $D_\alpha$ and $D_\beta$ such that
\[
 \frac{D_\beta}{n^\beta}\leq \pr_x\left(\tau_0>n\right)\leq \frac{D_\alpha}{n^\alpha}.
\]
\end{thm}
\begin{proof}
Let $\alpha$ and $\beta$ such that $\alpha<1-\theta<\beta$. We apply Theorem \ref{bound} and then there exists $A>0$, such that for all $x>A$, there exist $C_\alpha>0$ and $C_\beta>0$ such that
\[
\frac{C_\beta}{n^\beta}\leq \pr_x\left(\tau_A>n\right)\leq \frac{C_\alpha}{n^\alpha}.
\]
Since $\{0\} \subset [0,A]$, we obtain $\pr_x(\tau_A >n) \leq \pr_x(\tau_0>n)$ and then
\[
\frac{C_\beta}{n^\beta}\leq \pr_x\left(\tau_A>n\right)\leq \pr_x(\tau_0>n).
\]
Let $\left(T_\ell\right)_{\ell\geq 0}$ be a sequence of stopping times defined as below
\[
 T_\ell=\inf\{n \geq k_A+T_{\ell-1} \text{ such that } X_n\in [0,A]\},
\]
with $T_0=1$ and $k_A$ is the integer associated to $A$ such that \eqref{eteint} holds.
By \eqref{eteint}, we get 
\[
\pr_x(\tau_0>T_\ell)\leq (1-\varepsilon)^l.
\]
For $\alpha \in \left(0, 1-\theta\right)$, we get
\begin{align*}
 \esp_x\left(\tau_0^\alpha\right)&=\sum_{\ell=0}^{\infty} \esp_x\left(\mathds{1}_{\left\{T_\ell<\tau_0\leq T_{\ell+1}\right\}}\tau_0^\alpha\right) \leq \sum_{\ell=0}^{\infty}\esp_x\left(\mathds{1}_{\left\{T_\ell<\tau_0\right\}}T_{\ell+1}^\alpha\right)\\
& \leq \sum_{\ell=0}^{\infty}\esp_x\left(\mathds{1}_{\left\{T_\ell<\tau_0\right\}}\left(T_\ell+k_A+\left(T_{\ell+1}-k_A-T_\ell\right)\right)^\alpha\right)\\
& \leq \sum_{\ell=0}^{\infty}\esp_x\left(\mathds{1}_{\left\{T_\ell<\tau_0\right\}}\left(T_\ell^\alpha+k_A^\alpha+\left(T_{\ell+1}-k_A-T_\ell\right)^\alpha\right)\right).
\end{align*}
Let $\tau_{A,k_A}=\inf\{n \geq k_A \text{ such that } X_n \in [0,A]\}$.\\
Since $\left(T_{\ell+1}-k_A-T_\ell\right)^\alpha \leq \esp_{X_{T_\ell+k_A}}\left(\tau_A^\alpha\right)$, then by induction we obtain
\begin{align*}
 \esp_x\left(\tau_0^\alpha\right)& \leq \sum_{\ell=0}^{\infty}\esp_x\left(\mathds{1}_{\left\{T_\ell<\tau_0\right\}}\left(T_0^\alpha+\ell k_A^\alpha+\sum_{i=0}^\ell\esp_{X_{T_{i+k_A}}}\left(\tau_A^\alpha\right)\right)\right)\\
& \leq \esp_x\left(T_0^\alpha\right)+\sum_{\ell=0}^{\infty}\left(1-\varepsilon\right)^\ell \ell k^\alpha+\esp_x\left(\sum_{\ell=0}^{\infty}\mathds{1}_{\left\{T_\ell<\tau_0\right\}}\sum_{i=0}^\ell\esp_{X_{T_{i+k_A}}}\left(\tau_A^\alpha\right)\right)\\
& \leq \esp_x\left(T_0^\alpha\right)+\sum_{\ell=0}^{\infty}\left(1-\varepsilon\right)^\ell \ell k_A^\alpha+\sum_{\ell=0}^{\infty}\left(1-\varepsilon\right)^\ell \sup_{y \in [0,A]} \esp_y\left(\tau_{A,k_A}^\alpha\right)\\
& < \infty.
\end{align*}
We obtain the expected upper bound for $\pr_x\left(\tau_0>n\right)$ by Chebyshev's inequality.
\end{proof}

\subsection{A non-markovian example}
Let $(X_n)_{n\in\N}$ be a process defined by
\[
 X_{n+1}=X_n+1+K\varepsilon_n \sqrt{R_n}
\]
where $(\varepsilon_n)_{n\in\N}$ is a sequence of i.i.d. random variables such that for all $n\in \N$, $\pr(\varepsilon_n=-1)=\pr(\varepsilon_n=1)=\frac{1}{2}$, $K>2$ and $R_n$ defined as follows : 
\begin{itemize}
 \item Let $(N_n)_{n\in\N}$ be a sequence of independent integer-valued random variables such that $\forall i \in \{0, \ldots, n\}$, $\pr(N_n=i)=\frac{1}{n+1}$.
\item Let $(U_n)_{n\in\N}$ be a sequence of i.i.d. integer-valued random variables such that $\pr(U_n=0)=\pr(U_n=1)=\frac{1}{2}$.
\end{itemize}
We also assume that the random sequences $(N_n)_{n\in\N}$, $(U_n)_{n\in\N}$ and $(\varepsilon_n)_{n\in\N}$ are independent.\\
Let 
\[
R_n=U_n \frac{X_n^2}{X_n+X_{N_n}}+(1-U_n)\frac{X_nX_{N_n}}{X_n+X_{N_n}}. 
\]
If there exists $n\in \N$ such that $X_n \leq 0$, then for all $k\in \N$, $X_{n+k}=0$.\\
By construction, $(X_n)_{n\in\N}$ is not a Markov chain of any order.
Let us check that $(X_n)_{n\in\N}$ satisfies the stochastic difference equation $X_{n+1}=X_n+g(X_n)+\xi_n$ with $\esp\left(\xi_n \big| \F\right)=0$ and $\esp\left(\xi_n^2 \big| \F\right)=\sigma^2\left(X_n\right)$.
Let $\xi_n=\varepsilon_n K \sqrt{R_n}$. By independence, one has immediatly $\esp\left(\xi_n \big| \F\right)=0$. A short computation gives 
\begin{align*}
\esp\left(\xi_n^2 \big| \F\right)& =\frac{K^2}{2(n+1)}\,\underset{k=0}{\overset{n}{\sum}}\frac{X_n^2+X_nX_k}{X_n+X_k}\\
& =\frac{K^2}{2}X_n.
\end{align*}
Thus, $\theta=\frac{4}{K^2}$. If $K>2$, then we know that $\pr(\{X_n \underset{n \rightarrow \infty}{\longrightarrow} \infty\})=0$ and we can apply Theorem \ref{bound} and get lower and upper bound of tail of the hitting-time of $X_n$ in a compact set $[0,A]$.
\begin{prop}
Assume that $K>2$. For all $\alpha$ and $\beta$ such that $\alpha<1-4/K^2<\beta$, there exists $A>0$ such that for all $x>A$ there exist $C_\alpha>0$ and $C_\beta>0$ such that 
\[
 \frac{C_\alpha}{n^\alpha}\leq \pr_x\left(\tau_A>n\right)\leq \frac{C_\beta}{n^\beta}.
\]
\end{prop}
\appendix
\section{Proof of Lemma \protect{\ref{verif}}}\label{section:lemma-proof}
In this section, we turn to the proof of our key result, Lemma \ref{verif}.
\begin{proof}[Proof of Lemma \ref{verif}]
We first verify that $(Y_n)_{n\in\N}$ satisfies the first inequality of Lemma \ref{lemfonda}. Let $n\in \N$, then
\begin{align*}
 \esp\left(Y_{n+1}-Y_n\big| \F\right) & = \esp\left(\left(Y_{n+1}-Y_n\right)\left(\mathds{1}_{\{\xi_n\leq -g\left(X_n\right)-\varepsilon X_n\}}+\mathds{1}_{\{\xi_n > -g\left(X_n\right)-\varepsilon X_n\}}\right)\big| \F\right)\\
& \geq -Y_n \pr\left(\xi_n\leq -g\left(X_n\right)-\varepsilon X_n\right)\\
& \quad +\esp\left(\left(Y_{n+1}-Y_n\right)\mathds{1}_{\{\xi_n > -g\left(X_n\right)-\varepsilon X_n\}} \big| \F\right).
\end{align*}
We know that $\pr\left(\xi_n\leq -g\left(X_n\right)-\varepsilon X_n\right)\leq \frac{C\sigma^2\left(X_n\right)}{X_n^2}$ by Chebyshev's inequality. We use the Lagrange remainder of the Taylor series 
\[
G\left(X_{n+1}\right)=G\left(X_n\right)+\left(X_{n+1}-X_n\right)G'\left(X_n\right)+\frac{\left(X_{n+1}-X_n\right)^2}{2}\,G''\left(V_n\right),
\]
with $V_n$ between $X_n$ and $X_{n+1}$ :
\begin{align*}
\esp\left(Y_{n+1}-Y_n\big| \F\right) & \geq -\frac{C_1G\left(X_n\right)\sigma^2\left(X_n\right)}{X_n^2}\\
& \quad +\esp\left(\left(\frac{g\left(X_n\right)+\xi_n}{g\left(X_n\right)}\right)\mathds{1}_{\{\xi_n > -g\left(X_n\right)-\varepsilon X_n\}}\Big|\F\right)\\
& \quad -\esp\left(\frac{\left(g\left(X_n\right)+\xi_n\right)^2g'\left(V_n\right)}{2g^2\left(V_n\right)}\mathds{1}_{\{\xi_n > -g\left(X_n\right)-\varepsilon X_n\}}\Big|\F\right).
\end{align*}
We apply \ref{Gg/xitem} of Lemma \ref{lemmapre} on the first term and since $\esp\left(\xi_n \big|\F\right)=0$, then we can easily check that $\esp\left(\xi_n \mathds{1}_{\{\xi_n > -g\left(X_n\right)-\varepsilon X_n\}}\big|\F\right)>0$ and by \eqref{g'g} and \eqref{xg}, there exists $K>0$ such that $\frac{g'(V_n)}{g^2(V_n)}\leq \frac{K}{X_n g(X_n)}$  :
\begin{align*}
\esp\left(Y_{n+1}-Y_n\big| \F\right) & \geq -C_2-K\esp\left(\frac{\left(g\left(X_n\right)+\xi_n\right)^2}{2g\left(X_n\right)X_n}\Big|\F\right)\\
& \geq -C_2-K\frac{\sigma^2\left(X_n\right)}{2g\left( X_n\right) X_n} \geq -C_3.
\end{align*}
Thus, $(Y_n)_{n \in \N}$ verifies the first inequality of the Lemma \ref{lemfonda}.\\
We now check that there exists $D>0$ such that for all $n\in \N$
\[
 \esp\left(Y_{n+1}^2-Y_n^2\big| \F\right) \leq DY_n.
\]
First, note that
\[
 \esp\left(Y_{n+1}^2-Y_n^2\big| \F\right) \leq \esp\left(\left(Y_{n+1}^2-Y_n^2\right)\mathds{1}_{\{\xi_n>-g\left(X_n\right)-\varepsilon X_n\}}\big| \F\right).
\]
Once again, we use the Lagrange remainder of the Taylor series with $V_n$ between $X_n$ and $X_{n+1}$ : 
\begin{align*}
\esp&\left(Y_{n+1}^2-Y_n^2\big| \F\right) \leq 2G\left(X_n\right) + \esp\left(\frac{2\xi_n}{g\left(X_n\right)}G\left(X_n\right)\mathds{1}_{\{\xi_n>-g\left(X_n\right)-\varepsilon X_n\}}\Big| \F\right)\\
& \qquad + \esp\left(\frac{\left(\xi_n+g\left(X_n\right)\right)^2}{2}\left(\frac{2-2G\left(V_n\right)g'\left(V_n\right)}{g^2\left(V_n\right)}\right) \mathds{1}_{\{\xi_n>-g\left(X_n\right)-\varepsilon X_n\}}\Big| \F\right).
\end{align*}
Since $G(x)g'(x) \rightarrow \frac{1-\lambda}{\lambda}$, there exists $K_1>0$ such that for all $n\in \N$
\begin{align*}
\esp\left(Y_{n+1}^2-Y_n^2\big| \F\right)& \leq 2Y_n - 2\frac{Y_n}{g\left(X_n\right)}\,\esp\left(\xi_n \mathds{1}_{\{\xi_n<-g\left(X_n\right)-\varepsilon X_n\}}\Big| \F\right)\\
& \qquad +K_1\esp\left(\frac{\left(\xi_n+g\left(X_n\right)\right)^2}{g^2\left(V_n\right)} \mathds{1}_{\{\xi_n>-g\left(X_n\right)-\varepsilon X_n\}}\Big| \F\right).
\end{align*}
We have an upper bound for $\esp\left(\frac{\left(\xi_n+g\left(X_n\right)\right)^2}{g^2\left(V_n\right)} \mathds{1}_{\{\xi_n>-g\left(X_n\right)-\varepsilon X_n\}}\big| \F\right)$ by (A2) and by H\"older's inequality : 
\begin{align*}
& \esp\left(\frac{\left(\xi_n+g\left(X_n\right)\right)^2}{g^2\left(V_n\right)} \mathds{1}_{\{\xi_n>-g\left(X_n\right)-\varepsilon X_n\}}\Big| \F\right)\\
&\qquad \leq  \left( \esp\left(|\xi_n+g\left(X_n\right)|^{4+\delta}\big| \F\right)\right)^{\frac{2}{4+\delta}} \left(\esp\left(\frac{\mathds{1}_{\{\xi_n>-g\left(X_n\right)-\varepsilon X_n\}}}{g^{2+\frac{4}{2+\delta}}\left(V_n\right)} \Big| \F\right)\right)^{\frac{2+\delta}{4+\delta}}\\
&\qquad \leq K_2 \sigma^2\left(X_n\right) \left(\esp\left(\frac{V_n^{2+\frac{4}{2+\delta}}\mathds{1}_{\{\xi_n>-g\left(X_n\right)-\varepsilon X_n\}}}{X_n^{2+\frac{4}{2+\delta}}g^{2+\frac{4}{2+\delta}}\left(X_n\right)} \Big| \F\right)\right)^{\frac{2+\delta}{4+\delta}}\\
&\qquad \leq K_2 \frac{\sigma^2\left(X_n\right)}{g^2(X_n)}\left(\esp\left(\frac{X_{n+1}^{2+\frac{4}{2+\delta}}\mathds{1}_{\{\xi_n>-g\left(X_n\right)-\varepsilon X_n\}}}{X_n^{2+\frac{4}{2+\delta}}} \Big| \F\right)\right)^{\frac{2+\delta}{4+\delta}}\\
&\qquad \leq K_3 \frac{\sigma^2\left(X_n\right)}{g^2(X_n)},
\end{align*}
since $2+\frac{4}{2+\delta}\leq 4+\delta$ and $\esp\left(X_{n+1}^{2+\frac{4}{2+\delta}}\Big| \F \right)\leq K_4 X_n^{2+\frac{4}{2+\delta}}$.\\
Therefore, 
\begin{align*}
\esp\left(Y_{n+1}^2-Y_n^2\big| \F\right) & \leq 2Y_n + 2\frac{Y_n}{g\left(X_n\right)}K X_n \pr\left(\xi_n<-g\left(X_n\right)-\varepsilon X_n \big| \F \right)\\
& \qquad +K_5\frac{\sigma^2\left(X_n\right)}{g^2(X_n)}.
\end{align*}
Since $\pr\left(\xi_n<-g\left(X_n\right)-\varepsilon X_n \big| \F \right)\leq \frac{\sigma^2(X_n)}{\varepsilon^2X_n^2}$,
\begin{align*}
\esp\left(Y_{n+1}^2-Y_n^2\big| \F\right)& \leq 2Y_n + 2\frac{Y_n}{g\left(X_n\right)}K X_n \frac{\sigma^2\left(X_n\right)}{\varepsilon^2 X_n^2}+K_5 \frac{\sigma^2\left(X_n\right)}{g^2\left(X_n\right)}\\
& \leq  K_6 Y_n +K_7 \frac{X_n}{g\left(X_n\right)}\\
& \leq DY_n,
\end{align*}
by \ref{Gg/xitem} of Lemma \ref{lemmapre}, thus $(Y_n)_{n\in\N}$ satisfies the assumptions of the Lemma \ref{lemfonda}.
\end{proof}
 \begin{rem}\label{delta2}
  We can now explain why we need Assumption (A3) : since we can take $g(x)=1/x$, we can have $G(x)=x^2$ and then $Y_n^2=X_n^4$, so we need the existence of the fourth moment of $\xi_n$. 
 \end{rem}
\section{Auxiliary results}
In this last section, we recall some results from \cite{aspandiiarov3} that we applied above.\\
We recall that $\mathcal{A}$ is the set of positive function $f$ such that there exists a positive constant $A_f$ such that 
\[
 \underset{x \rightarrow \infty}{\limsup} \frac{f(2x)}{f(x)}\leq A_f.
\]
\begin{thm}[\cite{aspandiiarov3}, Theorem 2]\label{aspan2}
Let $\left(X_n\right)_{n \in \N}$ be an $\F$-adapted stochastic process taking values in an unbounded subset of $\R^+$. Let $f \in \mathcal{A}$ be an ultimately convex function. Suppose there exist positive constants $A_0,\varepsilon$ such that $(f(X_{n \wedge \tau_{A_0}}))_{n \in \N}$ is a supermartingale and for any $n \in \N$, on the event $\{ \tau_{A_0}>n\}$,
\[
\esp(f(X_{n+1})-f(X_n) \big| \F)\leq -\varepsilon f'(X_n).
\]
Then, there exists a positive constant $c$ such that for all $x\geq A_0$,
\[
\esp_x(f(\tau_{A_0}))\leq c f(x).
\]
\end{thm}
For all real valued functions $h$, let $\mathcal{B}_h$ be the set of  positive functions $f \in \mathcal{C}^2(0, \infty)$ ultimately concave, such that $\lim_{x \rightarrow \infty}f(x)=\infty$, $\lim_{x \rightarrow \infty}f'(x)=0$, and such that the integral
\[
\displaystyle{\int_1^\infty \frac{f'(x)\dd x}{h\circ r(x)} \quad \text{ converges,}}
\]
with $r(x)=\sup\{y \geq A, f'(x)=h'(y)\}$.
\begin{thm}[\cite{aspandiiarov3}, Theorem 3]\label{aspan3}
Let $\left(X_n\right)_{n\in \N}$ be an $\F$-adapted stochastic process taking values in an unbounded subset of $\R^+$. Let $h\in \mathcal{C}^1\left([0,\infty)\right)$ be a real-valued function such that $h'$ decreases in a neighborhood of $\infty$ and $h'(x) \rightarrow 0$ as $x \rightarrow \infty$. Suppose there exist positive constants $A_0,\varepsilon$ such that $h$ increases on $[A_0,\infty)$ and for any $n \in \N$, on the event $\{ \tau_{A_0}>n\}$,,
\[
\esp(h(X_{n+1})-h(X_n))\leq -\varepsilon h'(X_n).
\]
Then, for any $f \in \mathcal{B}_h$, there exist positive constants $c,A\geq A_0$ such that for all $x\geq A_0$,
\[
\esp_x(f(\tau_A))\leq c h(x).
\]
\end{thm}
We now recall a proposition from \cite{aspandiiarov3} which gives a link between integrability of hitting times of a finite set and of first return times to the initial state.\\
\begin{prop}[Proposition 1, \cite{aspandiiarov3}]\label{hittingreturn}
 Let $F$ be a finite subset of $\mathcal{X}$, $\tau_{F}=\inf \{ n>0, X_n\in F\}$ the hitting time of $F$ and $\tau=\inf\{ n>0, X_n=X_0\}$ be the first return time.\\
i) If for any $z\in F$,
\[
 \esp_z(\tau_F)<\infty,
\]
then for any $z \in F$, $\esp_z(\tau)<\infty$.\\
ii) If for some $z_0 \in F$, we have $\esp_{z_0}(\tau_F)=\infty$, then $\esp_{z_0}(\tau)=\infty$.
\end{prop}
The following theorem gives the speed of convergence to the invariant measure of probability of $(X_n)_{n\in\N}$ in the recurrent positive case. We first introduce two sets of positive functions.\\
Let $\mathcal{G}$ be the set of positive functions $f$ such that there exists a positive function $h$ such that $h(x)\rightarrow 0$ as $x \rightarrow \infty$ and a positive constant $c$ such that for any positive $m\geq 1$, $x_1\geq 1, \ldots, x_m \geq 1$, 
\[
 f\left(\sum_{k=1}^mx_k\right)\leq c e^{mh(m)}\sum_{k=1}^mf(x_k).
\]
Let $\mathcal{G}'$ be the set of non decreasing in a neighborhood of infinity functions $f$ such that $\ln(f(x))/x$ is non increasing in a neighborhood of infinity and tends to zero when $x$ tends to infinity.
\begin{thm}[Theorem 3, \cite{aspandiiarov3}]\label{fphi}
Let $f\in \mathcal{G}$ such that $f' \in \mathcal{G}'$. Suppose there exists a positive function $\phi$ defined on $\mathcal{X}$ such that for all $s \in \mathcal{X}\setminus F$,
\[
 \esp_s(f(\tau_F))\leq \phi(s),
\]
and also that for all $z \in F$, $\esp_z(\phi(X_1))<\infty$. Then, for any initial distribution $\nu$ on $\mathcal{X}$ such that 
\[
 \esp_\nu(f'(\tau_F))<\infty,
\]
we have
\[
\underset{n \rightarrow \infty}{\lim}f'(n) \sum_{i \in \mathcal{X}} \sum_{j \in \mathcal{X}} \nu(i) |P^n(i,j)-\pi(j)|=0.
\]
\end{thm}

\noindent \textbf{Acknowledgement.}  \\
The author is very grateful to Vincent Bansaye, Jean-Ren\'e Chazottes and Eric Moulines for many helpful discussions on the subject of this paper.
The author acknowledges partial support by the ``Chaire Mod\'elisation Math\'ematique et Biodiversit\'e'' of Veolia Environnement - \'Ecole Polytechnique - Museum National d'Histoire Naturelle - Fondation X.
\bibliographystyle{plain}
\bibliography{synthese}
\end{document}